\documentclass{amsart}
\def\Z{{\mathbb{Z}}}

\def\S{{\mathbb{S}}}

\def\St{\operatorname{St}}
\def\Lk{\operatorname{Lk}}
\newtheorem{Theorem}{Theorem}[section]

\newtheorem{Lemma}[Theorem]{Lemma}

\theoremstyle{definition}
\newtheorem{Definition}[Theorem]{Definition}

\newtheorem*{Problem}{Problem}

\theoremstyle{remark}
\newtheorem*{Remark}{Remark}

\begin{document}
\sloppy
\title{Reconstructible graphs, simplicial flag complexes of homology manifolds 
and associated right-angled Coxeter groups}
\author{Tetsuya Hosaka} 
\address{Department of Mathematics, 
Shizuoka University, Suruga-ku, Shizuoka 422-8529, Japan}
\date{January 31, 2015}
\email{hosaka.tetsuya@shizuoka.ac.jp}
\thanks{
Partly supported by the Grant-in-Aid for Young Scientists (B), 
The Ministry of Education, Culture, Sports, Science and Technology, Japan.
(No.\ 25800039).}
\begin{abstract}
In this paper, 
we investigate a relation between 
finite graphs, simplicial flag complexes and right-angled Coxeter groups, 
and we provide a class of reconstructible finite graphs.
We show that 
if $\Gamma$ is a finite graph which is the 1-skeleton 
of some simplicial flag complex $L$ which is 
a homology manifold of dimension $n \ge 1$, 
then the graph $\Gamma$ is reconstructible.
\end{abstract}
\maketitle
%
%%%%%%%%%%%%%
% Section 1 %
%%%%%%%%%%%%%
\section{Introduction}

In this paper, 
we investigate a relation between 
finite graphs, simplicial flag complexes and right-angled Coxeter groups, 
and we provide a class of reconstructible finite graphs.
This paper treats only ``simplicial'' graphs.
We show that 
if $\Gamma$ is a finite graph which is the 1-skeleton 
of some simplicial flag complex $L$ which is 
a homology manifold of dimension $n \ge 1$, 
then the graph $\Gamma$ is reconstructible.

A graph $\Gamma$ is said to be {\it reconstructible}, 
if any graph $\Gamma'$ with the following property~$(*)$ 
is isomorphic to $\Gamma$.
\begin{enumerate}
\item[$(*)$] Let $S$ and $S'$ be the vertex sets of $\Gamma$ and $\Gamma'$ respectively.
Then there exists a bijection $f:S\rightarrow S'$ such that 
the subgraphs $\Gamma_{S-\{s\}}$ and $\Gamma'_{S'-\{f(s)\}}$ are isomorphic for any $s\in S$, 
where $\Gamma_{S-\{s\}}$ and $\Gamma'_{S'-\{f(s)\}}$ 
are the full subgraphs of $\Gamma$ and $\Gamma'$ whose vertex sets are 
$S-\{s\}$ and $S'-\{f(s)\}$ respectively.
\end{enumerate}

The following open problem is well-known as the Reconstruction Conjecture.

\begin{Problem}[{Reconstruction Conjecture}]
Every finite graph with at least three vertices will be reconstructible?
\end{Problem}

Some classes of reconstructible graphs are known 
(cf.\ \cite{BoHe}, \cite{Ke}, \cite{Ma}, \cite{Mc}, \cite{Mc2}, \cite{Ni}) 
as follows:
Let $\Gamma$ be a finite graph with at least three vertices.
\begin{enumerate}
\item[${\rm (i)}$] If $\Gamma$ is a regular graph, then it is reconstructible.
\item[${\rm (ii)}$] If $\Gamma$ is a tree, then it is reconstructible.
\item[${\rm (iii)}$] If $\Gamma$ is not connected, then it is reconstructible.
\item[${\rm (iv)}$] If $\Gamma$ has at most 11 vertices, then it is reconstructible.
\end{enumerate}

Our motivation to consider 
graphs of the 1-skeletons of some simplicial flag complexes 
comes from the following idea 
on right-angled Coxeter groups and their nerves.

Details of Coxeter groups and Coxeter systems are 
found in \cite{Bo}, \cite{Br} and \cite{Hu}, and 
details of flag complexes, nerves, Davis complexes and their boundaries are 
found in \cite{D1}, \cite{D2} and \cite{Mo}.

Let $\Gamma$ be a finite graph and let $S$ be the vertex set of $\Gamma$.
Then the graph $\Gamma$ uniquely determines a finite simplicial flag complex $L$ 
whose 1-skeleton $L^{(1)}$ coincide with $\Gamma$.
Here a simplicial complex $L$ is a {\it flag complex}, if the following condition holds:
\begin{enumerate}
\item[$(**)$] For any vertex set $\{s_0,\dots,s_n\}$ of $L$, 
if $\{s_i,s_j\}$ spans 1-simplex in $L$ for any $i,j\in\{0,\dots,n\}$ with $i\neq j$ 
then the vertex set $\{s_0,\dots,s_n\}$ spans $n$-simplex in $L$.
\end{enumerate}

Also every finite simplicial flag complex $L$ uniquely determines 
a right-angled Coxeter system $(W,S)$ whose nerve $L(W,S)$ coincide with $L$ 
(cf.\ \cite{Bes0}, \cite{D1}, \cite{D2}, \cite{Dav0}, \cite{Dr}).
Here for any subset $T$ of $S$, 
$T$ spans a simplex of $L$ if and only if 
the parabolic subgroup $W_T$ generated by $T$ is finite 
(such a subset $T$ is called a {\it spherical subset of $S$}).

Moreover it is known that every right-angled Coxeter group $W$ 
uniquely determines its right-angled Coxeter system $(W,S)$ up to isomorphisms 
(\cite{Ra}, \cite{Ho00}).

By this corresponding, we can identify 
a finite graph $\Gamma$, 
a finite simplicial flag complex $L$, 
a right-angled Coxeter system $(W,S)$ and 
a right-angled Coxeter group $W$.

Let $\Gamma$ and $\Gamma'$ be finite graphs, 
let $L$ and $L'$ be the corresponding flag complexes, 
let $(W,S)$ and $(W',S')$ be the corresponding right-angled Coxeter systems, and 
let $W$ and $W'$ be the corresponding right-angled Coxeter groups, respectively.
Then the following statements are equivalent:
\begin{enumerate}
\item[(1)] $\Gamma$ and $\Gamma'$ are isomorphic as graphs;
\item[(2)] $L$ and $L'$ are isomorphic as simplicial complexes;
\item[(3)] $(W,S)$ and $(W',S')$ are isomorphic as Coxeter systems;
\item[(4)] $W$ and $W'$ are isomorphic as groups.
\end{enumerate}

Also, for any subset $T$ of the vertex set $S$ of the graph $\Gamma$, 
the full subgraph $\Gamma_T$ of $\Gamma$ with vertex set $T$ corresponds 
the full subcomplex $L_T$ of $L$ with vertex set $T$, 
the parabolic Coxeter system $(W_T,T)$ generated by $T$, and 
the parabolic subgroup $W_T$ of $W$ generated by $T$.

Hence we can consider the reconstruction problem 
as the problem on simplicial flag complexes and 
also as the problem on right-angled Coxeter groups.

Moreover, the right-angled Coxeter system $(W,S)$ associated by the graph $\Gamma$ 
defines the Davis complex $\Sigma$ which is a CAT(0) space 
and we can consider the ideal boundary $\partial \Sigma$ of the CAT(0) space $\Sigma$ 
(cf.\ \cite{Bes0}, \cite{Bes}, \cite{BH}, \cite{D1}, \cite{D2}, \cite{Dav0}, \cite{Dr}, 
\cite{Gr}, \cite{Gr0}, \cite{Mo}).
Then the topology of the boundary $\partial \Sigma$ is 
determined by the graph $\Gamma$, and 
the topology of $\partial \Sigma$ is also a graph invariant.

Based on the observations above, 
we can obtain the following lemma from results of 
F.~T.~Farrell \cite[Theorem~3]{F}, M.~W.~Davis \cite[Theorem~5.5]{Dav0} and \cite[Corollary~4.2]{Ho0} 
(we introduce details of this argument in Section~3).

\begin{Lemma}\label{Lemma:key}
Let $(W,S)$ be an irreducible Coxeter system where $W$ is infinite 
and let $L=L(W,S)$ be the nerve of $(W,S)$.
Then the following statements are equivalent:
\begin{enumerate}
\item[(1)] $W$ is a virtual Poincar\'{e} duality group.
\item[(2)] $L$ is a generalized homology sphere.
\item[(3)] $\tilde{H}^i(L_{S-T})=0$ for any $i$ and any non-empty spherical subset $T$ of $S$.
\end{enumerate}
\end{Lemma}

Here a generalized homology $n$-sphere is a polyhedral homology $n$-manifold with 
the same homology as an $n$-sphere $\S^n$ 
(cf.\ \cite[Section~5]{Dav0}, \cite{Dav1}, \cite[p.374]{Mu}, \cite{Qu}).
Also detail of (virtual) Poincar\'{e} duality groups is found 
in \cite{Br0}, \cite{Dav0}, \cite{Dav1}, \cite{F}.

In Lemma~\ref{Lemma:key}, we particularly note that 
the statement (3) is a local condition of $L$ which determines 
a global structure of $L$ as the statement (2).
From this observation, it seems that the following theorem holds.
(However the proof is not so obvious.)

\begin{Theorem}\label{Thm1}
Let $\Gamma$ be a finite graph with at least 3 vertices 
and let $(W,S)$ be the right-angled Coxeter system associated by $\Gamma$ 
(i.e.\ the 1-skeleton of the nerve $L(W,S)$ of $(W,S)$ is $\Gamma$).
If the Coxeter group $W$ is an irreducible virtual Poincar\'{e} duality group, 
then the graph $\Gamma$ is reconstructible.
Hence, 
\begin{enumerate}
\item[${\rm (i)}$] 
if $\Gamma$ is the 1-skeleton of some simplicial flag complex $L$ 
which is a generalized homology sphere, 
then the graph $\Gamma$ is reconstructible, and 
\item[${\rm (ii)}$] 
in particular, if $\Gamma$ is the 1-skeleton of 
some flag triangulation $L$ of some $n$-sphere $\S^n$ ($n\ge 1$), 
then the graph $\Gamma$ is reconstructible.
\end{enumerate}
\end{Theorem}

Here, based on this motivation, 
we investigate a finite graph which is the 1-skeleton of some simplicial flag complex 
which is a {\it homology manifold} as an extension of a generalized homology sphere, 
and we prove the following theorem.
(Hence as a corollary, we also obtain Theorem~\ref{Thm1}.)

\begin{Theorem}\label{Thm2}
Let $\Gamma$ be a finite graph with at least 3 vertices.
\begin{enumerate}
\item[${\rm (i)}$] 
If $\Gamma$ is the 1-skeleton of some simplicial flag complex $L$ 
which is a homology $n$-manifold ($n\ge 1$), 
then the graph $\Gamma$ is reconstructible.
\item[${\rm (ii)}$] 
In particular, if $\Gamma$ is the 1-skeleton of 
some flag triangulation $L$ of some $n$-manifold ($n\ge 1$), 
then the graph $\Gamma$ is reconstructible.
\end{enumerate}
\end{Theorem}

Here detail of homology manifolds is found in 
\cite[Section~5]{Dav0}, \cite{Dav1}, \cite[p.374]{Mu}, \cite{Qu}.

%%%%%%%%%%%%%
% Section 2 %
%%%%%%%%%%%%%
\section{Proof of Theorem~\ref{Thm2}}

We prove Theorem~\ref{Thm2}.

\begin{proof}[Proof of Theorem~\ref{Thm2}]
Let $\Gamma$ be a finite graph with at least 3 vertices 
which is the 1-skeleton 
of some simplicial flag complex $L$ which is 
a homology manifold of dimension $n \ge 1$.
Then we show that the graph $\Gamma$ is reconstructible.

Let $\Gamma'$ be a finite graph 
and let $L'$ be the finite simplicial flag complex associated by $\Gamma'$.
Also let $S$ and $S'$ be the vertex sets of the graphs $\Gamma$ and $\Gamma'$ respectively.

Now we suppose that the condition $(*)$ holds:
\begin{enumerate}
\item[$(*)$] There exists a bijection $f:S\rightarrow S'$ such that 
the subgraphs $\Gamma_{S-\{s\}}$ and $\Gamma'_{S'-\{f(s)\}}$ are isomorphic for any $s\in S$.
\end{enumerate}

To show that the graph $\Gamma$ is reconstructible, 
we prove that the two graphs $\Gamma$ and $\Gamma'$ are isomorphic, 
i.e., 
the two simplicial flag complexes $L$ and $L'$ associated by $\Gamma$ and $\Gamma'$ respectively 
are isomorphic.

Let $v_0 \in S$ and let $v'_0=f(v_0)$.
Then 
the two subgraphs $\Gamma_{S-\{v_0\}}$ and $\Gamma'_{S'-\{v'_0\}}$ are isomorphic 
by the assumption $(*)$, 
and 
the two subcomplexes $L_{S-\{v_0\}}$ and $L'_{S'-\{v'_0\}}$ are isomorphic.
Let $\phi$ be an isomorphism from $L_{S-\{v_0\}}$ to $L'_{S'-\{v'_0\}}$.

If for any $a \in \Lk(v_0,L)^{(0)}$, $\phi(a) \in \Lk(v'_0,L')^{(0)}$ 
then we obtain an isomorphism $\bar{\phi}:L\rightarrow L'$ 
from $\bar{\phi}|_{L_{S-\{v_0\}}}=\phi$ and $\bar{\phi}(v_0)=v'_0$ (since $\deg v_0=\deg v'_0$),
hence $L$ and $L'$ are isomorphic.

Now we suppose that there exists $a_0\in S-\{v_0\}$ such that 
$a_0\not\in \Lk(v_0,L)^{(0)}$ and 
$a'_0:=\phi(a_0)\in \Lk(v'_0,L')^{(0)}$.

Here if there does not exist $u'_0 \in S'-\St(a'_0,L')^{(0)}$, 
then $\St(a'_0,L')^{(0)}=S'$, 
where $\St(a'_0,L')$ means the closed star of $a'_0$ in $L'$.
Hence $[a'_0,b'] \in {L'}^{(1)}$ for any $b'\in S'-\{a'_0\}$.
Since $\deg a_0 = \deg a'_0$ and $|S|=|S'|$,
$[a_0,b] \in {L}^{(1)}$ for any $b\in S-\{a_0\}$.
This particularly implies $[a_0,v_0]\in L^{(1)}$.
This is a contradiction 
because it means $a_0\in \Lk(v_0,L)^{(0)}$.

Thus we suppose that there exists $u'_0 \in S'-\St(a'_0,L')^{(0)}$.

Let $u_0:=f^{-1}(u'_0)$.
Then by the assumption $(*)$, 
the two subcomplexes $L_{S-\{u_0\}}$ and $L'_{S'-\{u'_0\}}$ are isomorphic and 
let $\psi$ be an isomorphism from $L_{S-\{u_0\}}$ to $L'_{S'-\{u'_0\}}$.

Then 
\begin{align*}
\Lk(\psi^{-1}(a'_0),L_{S-\{u_0\}}) &\cong \Lk(a'_0,L'_{S'-\{u'_0\}}) \\ 
&\cong \Lk(a'_0,L'), 
\end{align*}
since $\psi$ is an isomorphism and $u'_0 \not\in \St(a'_0,L')$.
Also we obtain 
\begin{align*}
\St(\psi^{-1}(a'_0),L_{S-\{u_0\}}) &\cong \St(a'_0,L'_{S'-\{u'_0\}}) \\ 
&\cong \St(a'_0,L').
\end{align*}

Then 
$$ \St(a'_0,L'_{S'-\{v'_0\}}) {\underset{\neq}{\subset}} \St(a'_0,L') 
\cong \St(\psi^{-1}(a'_0),L_{S-\{u_0\}}). $$
Here we note that $\St(\psi^{-1}(a'_0),L_{S-\{u_0\}})$ is 
either 
\begin{enumerate}
\item[(a)] the closed star $\St(\psi^{-1}(a'_0),L)$ of the vertex $\psi^{-1}(a'_0)$ 
in the homology $n$-manifold $L$, or 
\item[(b)] $\St(\psi^{-1}(a'_0),L)-u_0$ where $u_0\in \Lk(\psi^{-1}(a'_0),L)$, 
\end{enumerate}
and also note that $\St(a'_0,L'_{S'-\{v'_0\}})=\St(a'_0,L')-v'_0$.
Hence we obtain that 
\begin{enumerate}
\item[(I)] $\St(a'_0,L'_{S'-\{v'_0\}})$ is isomorphic to some closed star 
deleted one or two vertices from its link in the homology $n$-manifold $L$.
\end{enumerate}

On the other hand, 
$$ 
\St(a'_0,L'_{S'-\{v'_0\}}) \cong \St(a_0,L_{S-\{v_0\}}) 
\cong \St(a_0,L), $$
since $\phi$ is an isomorphism and $a_0\not\in \St(v_0,L)$.
Here we note that $\St(a_0,L)$ is the closed star in the homology $n$-manifold $L$.
Hence we obtain that 
\begin{enumerate}
\item[(II)] $\St(a'_0,L'_{S'-\{v'_0\}})$ is isomorphic to some closed star in the homology $n$-manifold $L$.
\end{enumerate}

Then (I) and (II) imply the contradiction.
Indeed the following claim holds.

\vspace*{2mm}

{\bf Claim.} Let $A=\St(a)$ be a closed star of a vertex $a$ in a homology $n$-manifold and 
let $B=\St(b)-\{c_1,c_2\}$ be a closed star of a vertex $b$ 
deleted one or two vertices $\{c_1,c_2\} \subset \Lk(b)$ in a homology $n$-manifold.
Then the simplicial complexes $A$ and $B$ are not isomorphic.

\vspace*{1.5mm}

We first note that every triangulated homology $n$-manifold 
is a union of $n$-simplexes (\cite[Corollary~63.3(a)]{Mu}).
Hence $A=\St(a)$ and $\St(b)$ are unions of $n$-simplexes containing $a$ and $b$ respectively.
Then there exists an $n$-simplex $\sigma_0$ such that $c_1 \in \sigma_0 \subset \St(b)$.

Here if $c_1 \neq c_2$ then we can take $\sigma_0$ as $c_2 \not\in \sigma_0$.
Indeed if $c_1 \neq c_2$ and $c_2 \in \sigma_0$ then $[c_1,c_2] \subset \sigma_0$ and 
we can consider $(n-1)$-simplex $\tau$ as $\tau^{(0)}=\sigma_0^{(0)}-\{c_2\}$.
Then by \cite[Corollary~63.3(b)]{Mu}, 
there exist precisely two $n$-simplexes containing $\tau$ as a face.
Hence we can take an $n$-simplex $\sigma'_0$ containing $\tau$ as a face and $\sigma'_0\neq \sigma_0$.
Then $c_1 \in \sigma'_0 \subset \St(b)$ and $c_2 \not\in \sigma'_0$.
Hence in this case we retake $\sigma_0$ as $\sigma'_0$.

Now $\sigma_0$ is an $n$-simplex 
such that $c_1 \in \sigma_0 \subset \St(b)$ and 
if $c_1 \neq c_2$ then $c_2 \not\in \sigma_0$.
Let $\tau_0$ be the $(n-1)$-simplex as $\tau_0^{(0)}=\sigma_0^{(0)}-\{c_1\}$.
Then we note that $\tau_0 \subset \St(b)-\{c_1,c_2\}=B$.

Now we suppose that $A$ and $B$ are isomorphic and 
there exists an isomorphism $g:B\to A$.
Then $g(\tau_0)$ is an $(n-1)$-simplex in $A$.
By \cite[Corollary~63.3(b)]{Mu}, 
there exist precisely two $n$-simplexes $\bar{\sigma}_1$ and $\bar{\sigma}_2$ 
containing $g(\tau_0)$ as a face in $A$.
Then $g^{-1}(\bar{\sigma}_1)$ and $g^{-1}(\bar{\sigma}_2)$ are 
$n$-simplexes containing $\tau_0$ as a face in $B$, 
since $g:B\to A$ is an isomorphism.
Here $g^{-1}(\bar{\sigma}_1)$, $g^{-1}(\bar{\sigma}_2)$ and $\sigma_0$ 
are distinct $n$-simplexes containing $\tau_0$ as a face in $\St(b)$.
This contradicts to \cite[Corollary~63.3(b)]{Mu}.

Thus the simplicial complexes $A$ and $B$ are not isomorphic.

\vspace*{2mm}

Hence, there does not exist $a_0\in S-\{v_0\}$ such that 
$a_0\not\in \Lk(v_0,L)^{(0)}$ and $\phi(a_0) \in \Lk(v'_0,L')^{(0)}$, that is, 
for $a \in S-\{v_0\}$, 
$a \in \Lk(v_0,L)^{(0)}$ if and only if $\phi(a) \in \Lk(v'_0,L')^{(0)}$, 
since $\deg v_0=\deg v'_0$.
Hence the map $\bar{\phi}:S\rightarrow S'$ defined 
by $\bar{\phi}|_{S-\{v_0\}}=\phi$ and $\bar{\phi}(v_0)=v'_0$ 
induces an isomorphism of the two graphs $\Gamma$ and $\Gamma'$.

Therefore the graph $\Gamma$ is reconstructible.
\end{proof}

%%%%%%%%%%%%%
% Section 3 %
%%%%%%%%%%%%%
\section{Virtual Poincar\'{e} duality Coxeter groups and reconstructible graphs}

We introduce a relation of 
virtual Poincar\'{e} duality Coxeter groups and reconstructible graphs, 
which is our motivation of this paper.

\begin{Definition}[{cf.\ \cite{Br0}, \cite{Dav0}, \cite{Dav1}, \cite{F}}]
A torsion-free group $G$ is called an 
{\it $n$-dimensional Poincar\'{e} duality group}, 
if $G$ is of type FP and if 
$$ H^i(G;\Z G)\cong 
\left\{ 
\begin{array}{ll}
0  & (i\neq n) \\
\Z & (i=n).
\end{array}
\right.$$
Also a group $G$ is called a 
{\it virtual Poincar\'{e} duality group}, 
if $G$ contains a torsion-free subgroup of finite-index 
which is a Poincar\'{e} duality group.
\end{Definition}

On Coxeter groups and (virtual) Poincar\'{e} duality groups, 
the following results are known.

\begin{Theorem}[{Farrell \cite[Theorem~3]{F}}]\label{Thm:Farrell}
Suppose that $G$ is a finitely presented group of type FP, 
and let $n$ be the smallest integer such that 
$H^n(G;\Z G) \neq 0$.
If $H^n(G;\Z G)$ is a finitely generated abelian group, 
then $G$ is an $n$-dimensional Poincar\'{e} duality group.
\end{Theorem}

\begin{Remark}
It is known that every infinite Coxeter group $W$ 
contains some torsion-free subgroup $G$ of finite-index in $W$ 
which is a finitely presented group of type FP and 
$H^*(G;\Z G)$ is isomorphic to $H^*(W;\Z W)$.
Hence 
if $n$ is the smallest integer such that $H^n(W;\Z W) \neq 0$ and 
if $H^n(W;\Z W)$ is finitely generated (as an abelian group), 
then $W$ is a virtual Poincar\'{e} duality group of dimension $n$.
\end{Remark}

\begin{Theorem}[{Davis \cite[Theorem~5.5]{Dav0}}]\label{Thm:Davis}
Let $(W,S)$ be a Coxeter system.
Then the following statements are equivalent:
\begin{enumerate}
\item[(1)] $W$ is a virtual Poincar\'{e} duality group of dimension $n$.
\item[(2)] $W$ decomposes as a direct product $W=W_{T_0}\times W_{T_1}$ 
such that $T_1$ is a spherical subset of $S$ and 
the simplicial complex $L_{T_0}=L(W_{T_0},T_0)$ associated by $(W_{T_0},T_0)$ 
is a generalized homology $(n-1)$-sphere.
\end{enumerate}
\end{Theorem}

\begin{Theorem}[{\cite[Corollary~4.2]{Ho0}}]\label{Thm:Hos}
Let $(W,S)$ be an infinite irreducible Coxeter system, let $L=L(W,S)$ and let $0 \le i\in \Z$.
Then the following statements are equivalent:
\begin{enumerate}
\item[(1)] $H^i(W;\Z W)$ is finitely generated.
\item[(2)] $H^i(W;\Z W)$ is isomorphic to $\tilde{H}^{i-1}(L)$.
\item[(3)] $\tilde{H}^{i-1}(L_{S-T})=0$ 
for any non-empty spherical subset $T$ of $S$.
\end{enumerate}
Here $L_{S-T}=L(W_{S-T},S-T)$.
\end{Theorem}

We obtain the following lemma from results above.

\begin{Lemma}\label{Lem1}
Let $(W,S)$ be an irreducible Coxeter system where $W$ is infinite 
and let $L=L(W,S)$.
Then the following statements are equivalent:
\begin{enumerate}
\item[(1)] $W$ is a virtual Poincar\'{e} duality group.
\item[(2)] $L$ is a generalized homology sphere.
\item[(3)] $\tilde{H}^i(L_{S-T})=0$ for any $i$ and any non-empty spherical subset $T$ of $S$.
\end{enumerate}
\end{Lemma}

\begin{proof}
$(1)\Leftrightarrow (2)$: 
We obtain the equivalence of (1) and (2) from Theorem~\ref{Thm:Davis}, 
since $(W,S)$ is irreducible.

$(1)\Rightarrow (3)$: 
We obtain this implication from Theorem~\ref{Thm:Hos}, 
because if $W$ is a virtual Poincar\'{e} duality group 
then $H^i(W;\Z W)$ is finitely generated for any $i$.

$(3)\Rightarrow (1)$: 
Suppose that $\tilde{H}^i(L_{S-T})=0$ for any $i$ and 
any non-empty spherical subset $T$ of $S$.
Then by Theorem~\ref{Thm:Hos}, 
$H^{i+1}(W;\Z W)$ is finitely generated for any $i$.
Since $W$ is infinite, $H^{i_0}(W;\Z W)$ is non-trivial for some $i_0$ 
(cf.\ \cite{Br0}, \cite{GO}).
Hence by Theorem~\ref{Thm:Farrell}, 
$W$ is a virtual Poincar\'{e} duality group.
\end{proof}

We obtain Theorem~\ref{Thm1} from Theorem~\ref{Thm2}.
In particular, we obtain the following.

\begin{Theorem}\label{Thm1-2}
Let $\Gamma$ be a finite graph with at least 3 vertices 
and let $(W,S)$ be the right-angled Coxeter system associated by $\Gamma$.
If the Coxeter group $W$ is an irreducible virtual Poincar\'{e} duality group, 
then the graph $\Gamma$ is reconstructible.
\end{Theorem}

%%%%%%%%%%%%%%%%%%%%%%%%%%%%%%%%%%%%%
%             REFERENCES            %
%%%%%%%%%%%%%%%%%%%%%%%%%%%%%%%%%%%%%
%

%
\end{document}